\documentclass[10pt]{amsart}
\usepackage{multirow,bigdelim}
\usepackage{amssymb}
\usepackage{amscd}

\newtheorem{thm}{Theorem}[section]
\newtheorem{lem}[thm]{Lemma}
\newtheorem{cor}[thm]{Corollary}
\newtheorem{prop}[thm]{Proposition}

\theoremstyle{definition}
 \newtheorem{exa}[thm]{Example}

\begin{document}

\title{SUBGROUPS OF AN ABELIAN GROUP, RELATED IDEALS OF THE GROUP RING, 
AND QUOTIENTS BY THOSE IDEALS}
\maketitle


 
  \begin{center}
    {\small
    HIDEYASU KAWAI \\ \medskip
    General Education, Ishikawa College, National Institute of Technology,\\
    Ishikawa 929-0392, Japan\\
    {\tt kawai@ishikawa-nct.ac.jp}
    }
  \end{center}


\medskip
\noindent
{\small {\bf Abstract.} Let $RG$ be the group ring of an abelian group $G$ over a commutative ring $R$ with 
identity. An injection $\Phi$ from the subgroups of $G$ to the non-unit ideals of $RG$ 
is well-known. It is defined by $\Phi(N)=I(R,N)RG$ 
where $I(R,N)$ is the augmentation ideal of $RN$, and each ideal $\Phi(N)$ has a property : 
$RG/\Phi(N)$ is $R$-algebra isomorphic to $R(G/N)$. Let $\frak T$ be 
the set of non-unit ideals of $RG$. While the image of $\Phi$ is rather a small subset of 
$\frak T$, we give conditions on $R$ and $G$ for the image of $\Phi$ to have some 
distribution in $\frak T$. In the last section, we give criteria for choosing an element $x$ of $RG$ 
satisfying $RG/xRG$ is $R$-algebra isomorphic to $R(G/N)$ for a subgroup $N$ of $G$.

}

\bigskip

\section{Introduction}
In this paper, all groups are abelian and written multiplicatively, all rings 
are commutative with identity. We denote the group ring of a group $G$ over a 
ring $R$ by $RG$, and the augmentation ideal by $I(R,G)$. Recall that $I(R,G)$ is 
the kernel of the $R$-algebra homomorphism from $RG$ to $R$ induced by collapsing 
$G$ to $1$, and that $I(R,G)$ is a free $R$-module with $\{g-1|1\ne g\in G\}$ as 
a basis. Let $\frak T$ denote the set of non-unit ideals of $RG$, and let 
$\frak S$ denote the set of subgroups of $G$. We have the following two maps 
between $\frak T$ and $\frak S$.
 \begin{equation*}
  \Phi:\frak S \longrightarrow \frak T,\quad \Phi(N)=I(R,N)RG
 \end{equation*}
and
 \begin{equation*}
  \Psi:\frak T \longrightarrow \frak S, \quad \Psi(J)=G\cap(1+J).
 \end{equation*}
We know that 
 \begin{equation*}
  \Psi\circ\Phi=\text{identity}
 \end{equation*}
(see [2, Chap.2, Corollary 2.11]). So, 
 \begin{equation*}
  \text{$\Phi$ is injective and $\Psi$ is surjective.}
 \end{equation*}
This correspondence between $\frak S$ and $\frak T$ is well-known as we
can see in literature. Generally, $\Phi(\frak S)=\{\Phi(N)|N\in \frak S\}$ is 
rather a small subset of $\frak T$ since
 \begin{equation*}
  \Phi(N)\in \Psi^{-1}(N)
 \end{equation*}
and
 \begin{equation*}
  \frak T=\bigcup_{N\in \frak S}\Psi^{-1}(N) \quad (\text{disjoint union}).
 \end{equation*}
However, each element $\Phi(N)$ of $\Phi(\frak S)$ is the kernel of the $R$-algebra 
homomorphism $RG \rightarrow R(G/N)$ induced by the canonical surjection $G 
\rightarrow G/N$ (see [2, Chap.2, Proposition 2.10]). So
 \begin{equation*}
  RG/\Phi(N)\cong R(G/N),
 \end{equation*}
the quotient is a group ring over $R$ as well, that is why we take notice of this class of ideals.
We give conditions on $R$ and $G$ for $\Phi(\frak S)$ to have some distribution in $\frak T$.\par
In Section 2, first we discuss when the nilradical of $RG$ is contained in $\Phi(\frak S)$. A sufficient 
condition has been shown in [3]. It is proved to be necessary in this paper(Proposition 2.2). Next we 
also give a necessary and sufficient condition for $\Phi(\frak S)$ to contain the Jacobson radical
(Proposition 2.3). In Section 3, we consider conditions for $\Phi(\frak S)$ to include some classes of 
ideals as prime ideals, principal ideals and so on. In particular, it is shown that the set of non-unit 
principal ideals of a group ring cannot be included by $\Phi(\frak S)$ except for an almost trivial case
(Theorem 3.7). So, in Section 4, we give conditions for an element of the group ring of a cyclic group 
to generate an ideal contained in $\Phi(\frak S)$. \par
In the rest of this paper, we abbreviate $I(R,G)$ to $I(G)$ if it is not ambiguous, and 
denote the characteristic, the nilradical and the Jacobson radical of a ring $A$ by 
char $A$, $\frak N(A)$ and $\frak J(A)$, respectively.

\section{The nilradical and the Jacobson radical}
In this section, $R$ denotes a ring, and $G$ denotes a group. We set
 \begin{eqnarray*}
  G_p & = & \{g\in G\:|\:\text{the order of $g$ is  a power of a fixed prime number $p$}\}, \\
  {\rm supp}\:G & = & \{p\:|\:G_p\ne1\}.
 \end{eqnarray*}
Here two conditions $ \frak N(RG)\in\Phi(\frak S)$ 
and $ \frak J(RG)\in\Phi(\frak S)$ are investigated. \par
By definition of augmentation ideals, $\Phi(1)=I(1)RG=0$. We know that $\frak N(RG)=\Phi(1)$, 
namely $RG$ is reduced if and only if $R$ is reduced and $p$ is not a zero divisor of $R$ 
for all $p\in{\rm supp}\:G$ ([2, Chap.3, Corollary 4.3]). So we exclusively discuss the case 
$\frak N(RG)=\Phi(N)$ for a subgroup $N\ne 1$, and use the following proposition repeatedly.

\begin{prop}[{[2, Chap.3, Proposition 3.5]},{[1]}]\label{Proposition 2.1}
Let $R$ be a ring, and let $G$ be a group. Then $I(G)\subseteq \frak N(RG)$ if and only if 
there exists  a prime number $p$ such that $G$ is a $p$-group and $p\in \frak N(R)$.
\end{prop}

Then we have

\begin{prop}\label{Proposition 2.2}
For group rings $RG$ the following conditions are equivalent.\\
$(1)$ $\frak N(RG)=\Phi(N)$ for a subgroup $N\ne 1$.\\
$(2)$ $R$ is reduced and ${\rm char}\:R=p$ for some $p\in{\rm supp}\:G$.\\
Moreover, if these conditions are satisfied, $\frak N(RG)=\Phi(G_p)$.
\end{prop}

\begin{proof}
(1)$\Rightarrow$(2). We note $I(N)\subseteq \frak N(RN)$ since $I(N)\subseteq 
\Phi(N)=\frak N(RG)$. By Proposition \ref{Proposition 2.1}, there exists a prime 
number $p$ such that $N$ is a $p$-group and $p\in \frak N(R)$. So $N\subseteq G_p$, 
accordingly $p\in{\rm supp}\:G$. Moreover, from the isomorphism
 \begin{equation*}
  RG/\frak N(RG)=RG/\Phi(N)\cong R(G/N),
 \end{equation*}
we can say $R(G/N)$ is reduced and so is $R$. Taking notice of $p\in\frak N(R)$, 
we have ${\rm char}\:R=p$. Now we can show $\frak N(RG)=\Phi(G_p)$ as follows :
the inclusion $N\subseteq G_p$ implies $\frak N(RG)=\Phi(N)\subseteq \Phi(G_p)$, 
and $\Phi(G_p)=I(G_p)RG\subseteq \frak N(RG)$ since $I(G_p)\subseteq \frak N(RG_p)$ 
by Proposition \ref{Proposition 2.1}.\par
(2)$\Rightarrow$(1).([3, Lemma 2.3]) Using the isomorphism
 \begin{equation*}
  RG/\Phi(G_p)\cong R(G/G_p)
 \end{equation*}
with a necessary and sufficient condition for a group ring to be reduced, which is mentioned 
at the top of this section, we can say that $RG/\Phi(G_p)$ is reduced. So, $\frak N(RG)\subseteq
\Phi(G_p)$. By Proposition \ref{Proposition 2.1} it holds that $I(G_p)\subseteq
\frak N(RG_p)$. Thus $\Phi(G_p)=I(G_p)RG\subseteq \frak N(RG)$, and $\frak N(RG)=\Phi(G_p)$. 
This completes the proof.
\end{proof}

Like the case of $\frak N(RG)$, we know a necessary and sufficient condition for the equality 
$\frak J(RG)=\Phi(1)$ (namely, $=0$) to hold (see [2, Chap.3, Corollary 4.6]). Hence it is enough 
to consider the case $\frak J(RG)=\Phi(N)$ for a subgroup $N\ne 1$. In addition, we may assume 
$G$ is torsion since $\frak J(RG)=\frak N(RG)$ if $G$ is not torsion ([2, Chap.3, Corollary 4.7]). 
The following is crucial for the rest of this section : if $G$ is torsion, then
 \begin{equation*}\tag{2.1}
  \frak J(RG)=\frak J(R)G+\{r(g-1)\:|\:g\in G_p,\: r\in \frak J(R):_R p\;\; \text{for some 
  $p\in{\rm supp}\:G$}\}RG
 \end{equation*}
([2, Chap.3, Theorem 4.5]). So we have

\begin{prop}\label{Proposition 2.3}
If $G$ is torsion, the following conditions are equivalent.\\
$(1)$ $\frak J(RG)=\Phi(N)$ for a subgroup $N\ne 1$.\\
$(2)$ $\frak J(R)=0$ and ${\rm char}\:R=p$ for some $p\in{\rm supp}\:G$.\\
Moreover, if these conditions are satisfied, $\frak J(RG)=\Phi(G_p)$.
\end{prop}

\begin{proof}
(1)$\Rightarrow$(2). Note that $RG$ is integral over $RN$. Hence for any maximal 
ideal $m$ of $RN$, there exists a maximal ideal $M$ of $RG$ such that $M\cap RN=m$. 
By definition of $\Phi$ and condition $(1)$, $I(N)\subseteq \Phi(N)\subseteq M$. Thus 
$I(N)\subseteq m$. Accordingly $I(N)\subseteq \frak J(RN)$. It follows that $N$ is a $p$-group 
and $p\in \frak J(R)$ for a prime $p$ by [1, Proposition 15 (i)], and that $p\in{\rm supp}\:G$. 
By virtue of (2.1), if $g\in G_p$, then $g-1\in \frak J(RG)$. Hence $G_p\subseteq 1+\frak J(RG)$, 
and
 \begin{equation*}
  G_p\subseteq G\cap(1+\frak J(RG))=\Psi(\frak J(RG))=\Psi(\Phi(N))=N.
 \end{equation*}
So we have $G_p=N$ and $\frak J(RG)=\Phi(G_p)$. Now suppose that there exists a prime 
$q$ distinct from $p$ in ${\rm supp}\:G$. We know that $q$ cannot be contained in any 
maximal ideal of $R$ since $p\in \frak J(R)$. Thus $\frak J(R):_R q=\frak J(R)$. Let $g$ be 
an element of $G$ whose order is $q$. Let $r\in \frak J(R)$. Then by (2.1)
 \begin{equation*}
  r(g-1)\in \frak J(RG)=\Phi(G_p)
 \end{equation*}
and we have the following correspondence
 \begin{eqnarray*}
  RG/\Phi(G_p) & \cong & R(G/G_p) \\
  0=\overline{r(g-1)} & \mapsto & r\bar{g}-r\bar{1}=0 
 \end{eqnarray*}

\noindent where $\overline{r(g-1)}=r(g-1)\:{\rm mod}\:\Phi(G_p)$, $\bar{g}=gG_p$, and $\bar{1}=G_p$. 
Thus $r=0$. Therefore $\frak J(R)=0$ and ${\rm char}\:R=p$. Next suppose 
${\rm supp}\:G=\{p\}$. Since $G=G_p$, $\frak J(RG)=\Phi(G)=I(G)$. We know that 
$\frak J(R)=\frak J(RG)\cap R$ ([2, Chap.3, Lemma 4.4]). Hence $\frak J(R)=0$ and 
${\rm char}\:R=p$.\par
(2)$\Rightarrow$(1). Use Lemma 2.4.
\end{proof}

\begin{lem}\label{Lemma 2.4}
Let $G$ be a group not necessarily torsion. If $\frak J(R)=0$ and ${\rm char}\:R=p$ for some 
$p\in {\rm supp}\:G$, then $\frak J(RG)=\Phi(N)$ for a subgroup $N\ne 1$.
\end{lem}

\begin{proof}
By [2, Chap.3, Corollary 4.7], we see $\frak J(RG)=\frak N(RG)$. Apply Proposition 
\ref{Proposition 2.2}.
\end{proof}

The converse of Lemma \ref{Lemma 2.4} does not hold in general. Using a group $G$, 
which is not torsion, we can show a counterexample to the converse of Lemma \ref{Lemma 2.4} 
as follows. Let $F$ be a field of characteristic $p>0$ and let $F[x]$ be the polynomial ring in 
one variable over $F$. Denote by $F[x]_{(x)}$ the localization of $F[x]$ at a prime ideal 
$xF[x]$. We define $R=F[x]_{(x)}\times F[x]_{(x)}$ and $G=C_{\infty}\times C_p$, where 
$C_{\infty}$ and $C_p$ denote an infinite cyclic group and a cyclic group of order $p$ 
respectively. Then $\frak J(RG)=\frak N(RG)$ ([2, Chap.3, Corollary 4.7]), $p\in{\rm supp}\:G$, 
${\rm char}\:R=p$ and $\frak N(R)=0$. We see $\frak J(RG)=\Phi(G_p)$ by Proposition 
\ref{Proposition 2.2}. However clearly $\frak J(R)\ne 0$.

\section{Prime ideals and principal ideals}

\begin{lem}\label{Lemma 3.1}
Let $R$ be a ring, and let $G$ be a group. If the maximal ideals of $RG$ are
images of $\Phi$, then there exists a prime number $p$ such that $R$ is
a field of characteristic $p$ and $G$ is a $p$-group.
\end{lem}

\begin{proof}
Let $P$ be a maximal ideal of $RG$. Then there exists a subgroup $N$ of $G$ such
that $P=\Phi(N)=I(N)RG$. We see $P\subseteq I(G)$ since $I(N)RG\subseteq I(G)$,
and thus $I(G)$ is the unique maximal ideal of $RG$, and $R$ is a field since 
$RG/I(G)\cong R$. By using [2, Chap.3, Theorem 4.8], we can conclude 
that $G$ is a $p$-group and the characteristic of $R$ is $p$ with some prime number $p$.
\end{proof}

From Proposition 2.1, we get the following.

\begin{lem}\label{Lemma 3.2}
Let $R$ be a field of characteristic $p>0$, and let $G$ be a $p$-group. Then 
$I(G)$ is the unique prime ideal of $RG$.
\end{lem}

Using the preceding  lemmas, we have a corollary.

\begin{cor}\label{Corollary 3.3}
Let $R$ be a ring, and let $G$ be a group. Then the following are equivalent.\\
$(1)$ The maximal ideals of $RG$ are contained in $\Phi(\frak S)$.\\
$(2)$ $R$ is a field of characteristic $p>0$ and $G$ is a $p$-group.\\
$(3)$ $I(G)$ is the unique prime ideal of $RG$.\\
$(4)$ The prime ideals of $RG$ are contained in $\Phi(\frak S)$.
\end{cor}

\begin{lem}\label{Lemma 3.4}
Let $R$ be a field of characteristic $p>2$, and let $G$ be a $p$-group with $G\ne 1$.
Then there exists a non-unit principal ideal of $RG$ which is not contained in $\Phi(\frak S)$.
\end{lem}

\begin{proof}
Let $g$ be an element of $G$ whose order is $p$, and let $H$ be the subgroup of $G$
generated by $g$. Then by definition
 \begin{equation*}
  \Phi(H)=I(H)RG=(g-1)RG.
 \end{equation*}
Here, if $(g-1)RG=(g-1)^2RG$, then $(g-1)RG=(g-1)^pRG=0$. So, we have $(g-1)RG
\supsetneqq (g-1)^2RG$. Now assume that there exists a subgroup $K$ of $G$ such that
 \begin{equation*}
  \Phi(K)=(g-1)^2RG.
 \end{equation*}
Since $\Phi(H)\supsetneqq \Phi(K)$ and $\Psi\circ\Phi=\text{identity}$ (see 1. Introduction), 
we have
 \begin{equation*}
  H=\Psi(\Phi(H))\supsetneqq \Psi(\Phi(K))=K.
 \end{equation*}
Hence $K=1$, accordingly $\Phi(K)=0$. Thus $(g-1)^2=0$, whereas $g^2-2g+1\ne 0$
because $g^2$, $g$, and $1$ are different elements of $G$ by the assumption that 
the order of $g$ is $p>2$. This contradiction shows that $(g-1)^2RG\not\in\Phi(\frak S)$.
\end{proof}

\begin{lem}\label{Lemma 3.5}
Let $R$ be a field of characteristic $2$, and let $G$ be a $2$-group. Suppose $G$ has
an element of order $4$. Then there exists a non-unit principal ideal of $RG$ which is not 
contained in $\Phi(\frak S)$.
\end{lem}

\begin{proof}
Let $g$ be an element of $G$, and suppose the order of $g$ is $4$. Define
$I=(g-1)^3RG$. If $(g-1)^2RG=(g-1)^3RG$, then $(g-1)^3RG=(g-1)^4RG=(g^4-1)RG=0$.
Thus $(g-1)^3=0$. However
 \begin{equation*}\tag{3.1}
  (g-1)^3=g^3-3g^2+3g-1=g^3+g^2+g+1\ne 0
  \label{2.6.1}
 \end{equation*}
since the order of $g$ is $4$. Hence $(g-1)^2RG\supsetneqq (g-1)^3RG=I$. Moreover we
see that $I\ne 0$ by \eqref{2.6.1}. So, we have
 \begin{equation*}\tag{3.2}
  0\subsetneqq I\subsetneqq (g-1)^2RG=(g^2-1)RG=\Phi(H)
  \label{2.6.2}
 \end{equation*}
where $H$ denotes the subgroup of $G$ generated by $g^2$.\par
Suppose there exists a subgroup $K$ of $G$ such that $I=\Phi(K)$. Rewriting \eqref{2.6.2}, 
we have
 \begin{equation*}
  \Phi(1)\subsetneqq \Phi(K)\subsetneqq \Phi(H)
 \end{equation*}
and get
 \begin{equation*}
  1\subsetneqq K\subsetneqq H
 \end{equation*}
by mapping through $\Psi$. This is contradictory to the fact that $H$ is 
generated by $g^2$ and the order of $g$ is $4$. Therefore $I\not\in\Phi(\frak S)$.
\end{proof}

\begin{lem}\label{Lemma 3.6}
Let $R$ be a field of characteristic $2$, and let $G$ be a $2$-group which has no element 
of order $4$. Suppose that the order of $G$ is greater than $2$. Then there exists a non-unit 
principal ideal of $RG$ which is not contained in $\Phi(\frak S)$.
\end{lem}

\begin{proof}
By the assumption on the order of $G$, we have two different elements $f_1, f_2 \in G\setminus 
\{1\}$ satisfying $f_1f_2\not\in \{1, f_1, f_2\}$. Define an element $x$ of $RG$ by 
$x=1+f_1+f_2+f_1f_2$. Clearly, $xRG\ne RG$ since $x\in I(G)$. Now we assume $xRG\in 
\Phi(\frak S)$. Then $xRG$ has an element $1+g$ for some $g\in G\setminus \{1\}$ by 
[2, Chap.3, Lemma 1.3]. Hence there is an element $y$ of $RG$ such that $xy=1+g$. Note that 
$x=(1+f_1)(1+f_2)$, where $(1+f_i)^2=0$ for $i=1,2$. So, we have $(1+f_i)(1+g)=(1+f_i)xy=0$ 
for $i=1,2$. Thus $1+f_1, 1+f_2\in {\rm ann}(1+g)$, where ${\rm ann}(1+g)$ denotes the annihilator 
of $(1+g)RG$. Using the fact that the order of $g$ is $2$ and [2, Chap.2, Proposition 2.18 (iv)], we get 
${\rm ann}(1+g)=(1+g)RG$. Therefore $x=(1+f_1)(1+f_2)\in (1+g)^2RG=0$, whereas $x\ne 0$. 
By this contradiction, we conclude $xRG\not\in \Phi(\frak S)$.
\end{proof}

As a result, we can show

\begin{thm}\label{Theorem 3.7}
Let $R$ be a ring, and let $G$ be a group with $G\ne 1$. Then the following are equivalent.\\
$(1)$ $\Phi(\frak S)=\frak T$.\\
$(2)$ Every non-unit principal ideal of $RG$ is contained in $\Phi(\frak S)$.\\
$(3)$ $R$ is a field of characteristic $2$ and $G$ is a group of order $2$.
\end{thm}

\begin{proof}
(2)$\Rightarrow$(3). We know $\Phi(N)\subseteq I(G)$ for any subgroup $N$ of $G$ by 
definition. So it follows from (2) that $J\subseteq I(G)$ for every non-unit principal ideal $J$ of 
$RG$. Hence if $x$ is an element of $RG$ then we have
 \begin{equation*}
  x\not\in {\rm U}(RG)\Leftrightarrow x\in I(G)
 \end{equation*}
where ${\rm U}(RG)$ denotes the unit group of $RG$. Thus $I(G)$ is the unique maximal ideal 
of $RG$. By Corollary \ref{Corollary 3.3}, we see that $R$ is a field of characteristic $p>0$ and 
$G$ is a $p$-group. By Lemma \ref{Lemma 3.4}, we get $p=2$. By Lemma \ref{Lemma 3.5}, 
we can say $G$ has no element of order 4. By Lemma \ref{Lemma 3.6}, we conclude 
that the order of $G$ is 2.\par
(3)$\Rightarrow$(1). We use the fact that $RG$ is isomorphic to
a quotient ring of the polynomial ring $R[x]$ in an indeterminate $x$ over $R$. So, 
using the isomorphism
 \begin{equation*}
  RG\cong R[x]/(x^2-1)=R[x]/((x-1)^2),
 \end{equation*}
we see $\frak T=\{0,I(G)\}$. This completes the proof.
\end{proof}

\newpage
\section{Conditions for an element of the group ring of a cyclic group to generate an ideal contained 
in $\Phi(\frak S)$}

As shown in Theorem \ref{Theorem 3.7}, there exist elements in a group ring each element of which generates an ideal 
not contained in $\Phi(\frak S)$ except for a particular case. So, we give conditions for an element 
of the group ring of a cyclic group to generate an ideal contained in $\Phi(\frak S)$.

\begin{lem}\label{Lemma 4.1}
Let $R$ be a field, and let $G$ be a finite cyclic group of order $m$ with a generator $g$. Suppose $x$ is an element 
of $RG$, and write $x=\sum_{i=0}^{m-1}r_ig^i \enskip (r_i\in R)$. Let $n$ be a positive integer with $n<m$. Then, 
$g^n-1\in xRG$ if and only if ${\rm rank}(A_x)={\rm rank}(\tilde{A}_{x,n})$, where {\rm rank(\quad)} denotes the rank of 
a matrix, and matrices $A_x$, $\tilde{A}_{x,n}$ are defined as follows:\\
\begin{equation*}
A_x=\begin{pmatrix}
        r_{m-1} & r_{m-2} & \cdots & r_1 & r_0 \\
        r_{m-2} & r_{m-3} & \cdots & r_0 & r_{m-1} \\
        r_{m-3} & r_{m-4} & \cdots & r_{m-1} & r_{m-2} \\
        \vdots   & \vdots   &            & \vdots & \vdots \\
        \vdots   & \vdots   &            & \vdots & \vdots \\
        r_2       & r_1        & \cdots & r_4 & r_3 \\
        r_1       & r_0        & \cdots & r_3 & r_2 \\
        r_0       & r_{m-1} & \cdots & r_2 & r_1
        \end{pmatrix},
\end{equation*}
\vspace{0.5cm}
\begin{equation*}
\tilde{A}_{x,n}=
                       \begin{array}{cc|ccccccccccc|ccc}
                        \cline{3-13}
                        \ldelim({9}{4pt}&&&&&&&&&&&&&0&\rdelim){9}{4pt}&\rdelim\}{4}{4pt}[$m-n$]\\
                       &&&&&&&&&&&&&\vdots&&\\
                       &&&&&&&&&&&&&0&&\\
                       &&&&&&\mbox{\smash{\Large $A_x$}}&&&&&&&1&&\\
                       &&&&&&&&&&&&&0&&\\
                       &&&&&&&&&&&&&\vdots&&\\
                       &&&&&&&&&&&&&0&&\\
                       &&&&&&&&&&&&&-1&&\\
                       \cline{3-13}
                       \end{array}.
\end{equation*}
\end{lem}

\begin{proof}
Suppose that $g^n-1=xy$ for an element $y$ of $RG$, and $y=\sum_{j=0}^{m-1}s_jg^j \enskip (s_j\in R)$. 
By calculating coefficients of $g^k$ ( $k=0,1, \ldots , m-1$ ) in $xy$, we have
\begin{equation}
 \begin{cases}\tag{4.1}
      r_{m-1}s_0+r_{m-2}s_1+\cdots+r_1s_{m-2}+r_0s_{m-1}=0 \\
      r_{m-2}s_0+r_{m-3}s_1+\cdots+r_0s_{m-2}+r_{m-1}s_{m-1}=0 \\
      \hdotsfor{1} \\
      r_ns_0+r_{n-1}s_1+\cdots+r_{n+2}s_{m-2}+r_{n+1}s_{m-1}=1 \\
      \hdotsfor{1} \\
      r_1s_0+r_0s_1+\cdots+r_3s_{m-2}+r_2s_{m-1}=0 \\
      r_0s_0+r_{m-1}s_1+\cdots+r_2s_{m-2}+r_1s_{m-1}=-1
 \end{cases}
\end{equation}
So, $g^n-1\in xRG$ if and only if the system (4.1) of linear equations for variables $s_0, s_1, \ldots , s_{m-1}$ 
has a solution. It is a well-known fact in linear algebra that the latter condition is equivalent to the condition 
${\rm rank}(A_x)={\rm rank}(\tilde{A}_{x,n})$.
\end{proof}

\newpage
We denote the greatest common divisor of positive integers $m$ and $n$ by $(m,n)$. And we denote the subgroup 
generated by an element $s$ of a group by $\langle s \rangle$.\par
\medskip
The following is easily shown.

\begin{lem}\label{Lemma 4.2}
Let $G$ be a finite cyclic group of order $m$, and let $g$ be a generator of $G$. Supposing $n$ is a positive 
integer with $n<m$ and $(m,n)=d$, then $\langle g^n \rangle=\langle g^d \rangle$.
\end{lem}

So,we can show
\begin{lem}\label{Lemma 4.3}
Let $R$ be a field, and let $G$ be a finite cyclic group of order $m$ with a generator $g$. Suppose $x$ is 
an element of $RG$, and write $x=\sum_{i=0}^{m-1}r_ig^i \enskip (r_i\in R)$. Let $n$ be a positive integer with 
$n<m$ and $(m,n)=d$. Then, $x\in (g^n-1)RG$ if and only if the following holds:
\begin{equation}
 \begin{cases}\tag{4.2}
      r_0+r_d+r_{2d}+\cdots+r_{(e-1)d}=0 \\
      r_1+r_{d+1}+r_{2d+1}+\cdots+r_{(e-1)d+1}=0 \\
      \hdotsfor{1} \\
      r_{d-1}+r_{2d-1}+r_{3d-1}+\cdots+r_{m-1}=0
 \end{cases}
\end{equation}
where $e=m/d$.
\end{lem}

\begin{proof}
We know that $I(R,\langle g^n \rangle)RG=(g^n-1)RG$ and $I(R,\langle g^d \rangle)RG=(g^d-1)RG$ by 
[2, Chap.3, Lemma 1.3]. 
Hence $(g^n-1)RG=(g^d-1)RG$ by Lemma \ref{Lemma 4.2}, therefore it suffices to show that 
$x\in (g^d-1)RG$ if and only if the condition (4.2) holds. It is clear that we have $x=(g^d-1)z$ with 
$z=\sum_{j=0}^{m-1}t_jg^j \enskip (t_j\in R)$ if and only if the following system of linear equations has a 
solution for $t_0, t_1, \ldots , t_{m-1}$ : \bigskip
\begin{equation}
 \begin{array}{ccccccccccccccccccc}\tag{4.3}
                     &\multicolumn{7}{l}{\overbrace{\hspace{15em}}^{\mbox{$m-d+1$}}}&&&&&&&&&&& \\
   \ldelim({14}{2pt}&-1&&&&&&1&&&&\rdelim){14}{2pt}&  \ldelim({14}{2pt}&t_0&\rdelim){14}{2pt}  &  &  \ldelim({14}{2pt}&r_0&\rdelim){14}{2pt} \\
                            &&-1&&&&&&1&\multicolumn{2}{c}{\text{\Huge 0}}& & &t_1& &  & &r_1& \\
                            &&&\ddots&&&&&&\ddots&& & &\vdots& &  & &\vdots& \\
                            &&&&-1&&&\multicolumn{2}{c}{\text{\Huge 0}}&&1& & &t_{d-1}& &  & &r_{d-1}& \\
                            &1&&&&-1&&&&&& & &t_d& &=& &r_d& \\
                            &&\ddots&&&&\ddots&&&&& & &\vdots& &  & &\vdots& \\
                            &&&\ddots&\multicolumn{2}{c}{\text{\Huge 0}}&&\ddots&&&& & &\vdots& &  & &\vdots& \\
                            &&&&\ddots&&&&\ddots&&& & &\vdots& &  & &\vdots& \\
                            &&\multicolumn{2}{c}{\text{\Huge 0}}&&\ddots&&&&\ddots&& & &\vdots& &  & &\vdots& \\
                            &&&&&&1&&&&-1& & &t_{m-1}& &  & &r_{m-1}& \\
                            &&&&&&\multicolumn{5}{l}{\underbrace{\hspace{10em}}_{\mbox{$d+1$}}}&&&&&&&&
 \end{array} \enskip .
\end{equation}
\newpage
The corresponding matrix of coefficients can also be expressed as follows :
\begin{equation*}
\left(
 \begin{array}{cccccc}
  \begin{array}{|ccc|}
  \hline
  && \\
  &\text{\large $-I_d$}& \\
  && \\
  \hline
  \end{array}
  &&&&
  \begin{array}{|ccc|}
  \hline
  && \\
  &\text{\large $\ I_d \ $}& \\
  && \\
  \hline
  \end{array}
  &
  \begin{array}{c}
  r_0 \\
  \vdots \\
  r_{d-1}
  \end{array} \\
  \begin{array}{|ccc|}
  \hline
  && \\
  &\text{\large $\ I_d \ $}& \\
  && \\
  \hline
  \end{array}
  &
  \begin{array}{|ccc|}
  \hline
  && \\
  &\text{\large $-I_d$}& \\
  && \\
  \hline
  \end{array}
  &&\text{\Huge 0}&&
  \begin{array}{c}
  \vdots \\
  \vdots \\
  \vdots
  \end{array} \\
  &
  \begin{array}{|ccc|}
  \hline
  && \\
  &\text{\large $\ I_d \ $}& \\
  && \\
  \hline
  \end{array}
  &
  \begin{array}{ccc}
  \ddots&& \\
  &\ddots& \\
  &&\ddots
  \end{array}
  &&&
  \begin{array}{c}
  \vdots \\
  \vdots \\
  \vdots
  \end{array} \\
  &&
  \begin{array}{ccc}
  \ddots&& \\
  &\ddots& \\
  &&\ddots
  \end{array}
  &
  \begin{array}{|ccc|}
  \hline
  && \\
  &\text{\large $-I_d$}& \\
  && \\
  \hline
  \end{array}
  &&
  \begin{array}{c}
  \vdots \\
  \vdots \\
  \vdots
  \end{array} \\
  \text{\Huge 0}&&&
  \begin{array}{|ccc|}
  \hline
  && \\
  &\text{\large $\ I_d \ $}& \\
  && \\
  \hline
  \end{array}
  &
  \begin{array}{|ccc|}
  \hline
  && \\
  &\text{\large $-I_d$}& \\
  && \\
  \hline
  \end{array}
  &
  \begin{array}{c}
  r_{m-d} \\
  \vdots \\
  r_{m-1}
  \end{array}
 \end{array}
\right) ,
\end{equation*}
where $I_d$ denotes the identity matrix of size $d$. Using elementary row operations, 
we can transform this matrix into the following
\begin{equation*}
\left(
 \begin{array}{cccccc}
  \begin{array}{|ccc|}
  \hline
  && \\
  &\text{\large $-I_d$}& \\
  && \\
  \hline
  \end{array}
  &&&&
  \begin{array}{|ccc|}
  \hline
  && \\
  &\text{\large $\ I_d \ $}& \\
  && \\
  \hline
  \end{array}
  &
  \begin{array}{c}
  r_0 \\
  \vdots \\
  r_{d-1}
  \end{array} \\
  &
  \begin{array}{|ccc|}
  \hline
  && \\
  &\text{\large $-I_d$}& \\
  && \\
  \hline
  \end{array}
  &&\text{\Huge 0}&
  \begin{array}{|ccc|}
  \hline
  && \\
  &\text{\large $\ I_d \ $}& \\
  && \\
  \hline
  \end{array}
  &
  \begin{array}{c}
  r_0+r_d \\
  \vdots \\
  r_{d-1}+r_{2d-1}
  \end{array} \\
  &&
  \begin{array}{ccc}
  \ddots&& \\
  &\ddots& \\
  &&\ddots
  \end{array}
  &&
  \begin{array}{c}
  \vdots \\
  \vdots \\
  \vdots
  \end{array}
  &
  \begin{array}{c}
  \vdots \\
  \vdots \\
  \vdots
  \end{array} \\
  &&&
  \begin{array}{|ccc|}
  \hline
  && \\
  &\text{\large $-I_d$}& \\
  && \\
  \hline
  \end{array}
  &
  \begin{array}{|ccc|}
  \hline
  && \\
  &\text{\large $\ I_d \ $}& \\
  && \\
  \hline
  \end{array}
  &
  \begin{array}{c}
  \sum_{i=0}^{e-2}r_{id} \\
  \vdots \\
  \sum_{i=0}^{e-2}r_{d-1+id}
  \end{array} \\
  \text{\Huge 0}&&&&&
  \begin{array}{c}
  \sum_{i=0}^{e-1}r_{id} \\
  \vdots \\
  \sum_{i=0}^{e-1}r_{d-1+id}
  \end{array}
 \end{array}
\right) .
\end{equation*}
Hence we see that the system (4.3) has a solution if and only if the condition (4.2) holds. 
This completes the proof.
\end{proof}

In the following statements, for a field $R$, ${\rm dim}_R(\quad)$ denotes the dimension of an $R$-vector space.

\begin{lem}\label{Lemma 4.4}
Let $R$ be a field, $G$ a finite cyclic group, and $x$ a nonzero element of $RG$. Then ${\rm dim}_R\thinspace xRG={\rm rank}(A_x)$, 
where $A_x$ is a matrix defined in {\rm Lemma \ref{Lemma 4.1}}.
\end{lem}

\begin{proof}
Suppose $G=\langle g \rangle$ with $|G|=m$, and write $x=\sum_{i=0}^{m-1}r_ig^i \enskip (r_i\in R)$.
We see that $xRG$ is generated over $R$ by $\{ g^j x\}_{0\leq j\leq m-1}$. Choosing $G$ as a basis of $RG$, 
we have a canonical isomorphism of $R$-vector spaces $RG\cong R^m$. This isomorphism maps the elements of 
$\{ g^j x\}_{0\leq j\leq m-1}$ as follows.
\begin{equation*}
 \begin{array}{ccc}
 x & \longmapsto & (r_{0\quad \ }, r_{1\quad \ }, r_{2\quad \ },\quad \ldots \quad, r_{m-2}, r_{m-1}) \\
 gx &\longmapsto & (r_{m-1}, r_{0\quad \ }, r_{1\quad \ },\quad \ldots \quad, r_{m-3}, r_{m-2}) \\ 
 g^2x &\longmapsto & (r_{m-2}, r_{m-1}, r_{0\quad \ },\quad \ldots \quad, r_{m-4}, r_{m-3}) \\
         & \vdots & \\
  g^{m-1}x &\longmapsto & (r_{1\quad \ }, r_{2\quad \ }, r_{3\quad \ },\quad \ldots \quad, r_{m-1}, r_{0\quad \ }) \\
 \end{array}
\end{equation*}
Thus we get ${\rm dim}_R\thinspace xRG={\rm rank}(A_x)$ as desired.
\end{proof}

\begin{prop}
Let $R$ be a field, and let $G$ be a finite cyclic group of order $m$ with a generator $g$. 
Suppose $x$ is a nonzero element of $RG$, and set $d=m-{\rm rank}(A_x)$, where $A_x$ is a matrix defined in 
{\rm Lemma 4.1}. Then the following are equivalent.\\
$(1)$ $xRG=\Phi(N)$ for a subgroup $N\ne 1$.\\
$(2)$ $d$ is a divisor of $m$ for which {\rm (4.2)} of {\rm Lemma} $\ref{Lemma 4.3}$ holds 
and ${\rm rank}(A_x)={\rm rank}(\tilde{A}_{x,d})$, where $\tilde{A}_{x,d}$ is a matrix defined in {\rm Lemma} $\ref{Lemma 4.1}$.\\
Moreover, if these are satisfied, $xRG=\Phi(\langle g^d \rangle)$, $RG/xRG\cong R(G/\langle g^d \rangle)$.
\end{prop}

\begin{proof}
(1)$\Rightarrow$(2). Assume $N=\langle g^n \rangle$ with $0<n<m$. Let $(m,n)=d^{\prime}$, then 
$\langle g^n \rangle = \langle g^{d^{\prime}} \rangle$ by Lemma \ref{Lemma 4.2}. Since $\Phi(\langle g^{d^{\prime}} \rangle)=
I(R, \langle g^{d^{\prime}} \rangle)RG=(g^{d^{\prime}}-1)RG$ by [2, Chap.3, Lemma 1.3], we have $xRG=\Phi(\langle 
g^{d^{\prime}} \rangle)=(g^{d^{\prime}}-1)RG$. Hence $RG/xRG=RG/\Phi(\langle g^{d^{\prime}} \rangle)\cong R(G/\langle 
g^{d^{\prime}} \rangle)$ (see Section 1). Therefore, ${\rm dim}_R\thinspace RG/xRG=|G/\langle g^{d^{\prime}} \rangle|=
d^{\prime}$. On the other hand, ${\rm dim}_R\thinspace RG/xRG=m-{\rm dim}_R\thinspace xRG$. Thus 
$d^{\prime}=m-{\rm rank}(A_x)$ by Lemma \ref{Lemma 4.4} and, consequently, $d^{\prime}=d$. Since $xRG=(g^d-1)RG$, 
from Lemma \ref{Lemma 4.1} and Lemma \ref{Lemma 4.3}, (4.2) of Lemma \ref{Lemma 4.3} holds and 
${\rm rank}(A_x)={\rm rank}(\tilde{A}_{x,d})$ for $d$.\par
(2)$\Rightarrow$(1). From Lemma \ref{Lemma 4.1} and Lemma \ref{Lemma 4.3}, we have $xRG=(g^d-1)RG$. 
Thus $xRG=I(R,\langle g^d \rangle)RG=\Phi(\langle g^d \rangle)$ with $\langle g^d \rangle \ne 1$.
\end{proof}

\begin{exa}
Let $\mathbb{F}_5=\mathbb{Z}/5\mathbb{Z}$, the prime field of order 5, and let $\alpha=2{\rm mod}5\in \mathbb{F}_5$. 
We denote by $C_{12}$ a cyclic group of order 12 with a generator $g$. Let $x$ be as follows : \\
\begin{equation*}
x=g+\alpha^3g^2+g^3+g^4+\alpha^3g^5+g^6+g^7+\alpha^2g^8+g^9+g^{10}+\alpha^3g^{11}\in \mathbb{F}_5C_{12}.
\end{equation*}
Then, ${\rm rank}(A_x)=8$, $d=|C_{12}|-{\rm rank}(A_x)=4$. Denoting $x=\sum_{i=0}^{11}r_ig^i \enskip (r_i\in \mathbb{F}_5)$, 
we have \\
\begin{equation*}
 \begin{array}{ccccc}
    r_0+r_4+r_8 & = & 0+1+\alpha^2 & = & 0, \\
    r_1+r_5+r_9 & = & 1+\alpha^3+1 & = & 0, \\
    r_2+r_6+r_{10} & = & \alpha^3+1+1 & = & 0, \\
    r_3+r_7+r_{11} & = & 1+1+\alpha^3 & = & 0,
 \end{array}
\end{equation*}
which shows (4.2) of Lemma \ref{Lemma 4.3} holds. Since ${\rm rank}(\tilde{A}_{x,4})=8$, we get $x\thinspace \mathbb{F}_5C_{12}=
\Phi(\langle g^4 \rangle)$ and $\mathbb{F}_5C_{12}/x\thinspace \mathbb{F}_5C_{12}\cong \mathbb{F}_5(C_{12}/\langle g^4 \rangle)$ 
as $\mathbb{F}_5$-algebras.
\end{exa}

For infinite cyclic groups, we have

\begin{prop}
Let $R$ be an integral domain, and let $G$ be an infinite cyclic group. Suppose $x$ is a nonzero element of $RG$. Then the following are 
equivalent.\\
$(1)$ $xRG=\Phi(N)$ for a subgroup $N\ne 1$.\\
$(2)$ $x=ug_1-ug_2$ for a unit $u$ of $R$ and two elements $g_1$, $g_2$ of $G$.\\
Moreover, if these are satisfied, $xRG=\Phi(\langle h \rangle)$, $RG/xRG\cong R(G/\langle h \rangle)$, where $h=g_1g_2^{-1}$.
\end{prop}

\begin{proof}
Since $G$ is a cyclic group, we see that (1) holds if and only if there exists an element $h$ of $G$ satisfying $(h-1)RG=xRG$. 
So, we replace (1) with the condition $\rm{(1)}^{\prime}$: $xRG=(h-1)RG$ for an element $h$ of $G$.\par
$\rm{(1)}^{\prime}$$\Rightarrow$(2). We have $h-1=xy$ for an element $y$ of $RG$, and $x=(h-1)z$ for an element $z$ of $RG$. 
Hence $h=xy+1=(h-1)zy+1=hyz-yz+1$. Thus $h-hyz=1-yz$, that is, $h(1-yz)=1-yz$. For each element $w$ of $RG$, define 
${\rm Supp}(w)$ as follows: if $w=\sum_{g\in G}r_gg(r_g\in R)$ then ${\rm Supp}(w)=\{~g\in G~|~r_g\ne 0~\}$. Using this notation, 
we can say that if $1-yz\ne 0$ then ${\rm Supp}(h(1-yz))\ne {\rm Supp}(1-yz)$ since $G$ is an infinite cyclic group. Thus $1-yz=0$, 
so $y$ is a unit of $RG$. Because $R$ is an integral domain and $G$ is torsion-free, $y$ is a trivial unit, that is, $y=rt$ for a unit $r$ 
of $R$ and an element $t$ of $G$ by [2, Chap.2, Proposition 2.23]. Recalling $xy=h-1$, we get $x=y^{-1}h-y^{-1}=r^{-1}(t^{-1}h)-r^{-1}t^{-1}$. 
If we set $u=r^{-1}$, $g_1=t^{-1}h$, and $g_2=t^{-1}$, then we have $x=ug_1-ug_2$ for a unit $u$ of $R$ and $g_1, g_2\in G$.\par
(2)$\Rightarrow$$\rm{(1)}^{\prime}$. Setting $h=g_1g_2^{-1}$, we get $x=uhg_2-ug_2=(h-1)ug_2$.
\end{proof}


\begin{thebibliography}{99}

 \bibitem{Con} I. G. Connell, On the group ring, Canad. J. Math. 15, 1963, 650-685.
 \bibitem{Kar} G. Karpilovsky, {\it Commutative Group Algebras}, Dekker, New York, 1983.
 \bibitem{Kaw} H. Kawai and N. Onoda, Commutative group algebras whose quotient 
 rings by nilradicals are generated by idempotents, Rocky Mt. J. Math. 41, 2011, 229-238.

\end{thebibliography}
\end{document}